\documentclass[11pt]{amsart}
\usepackage{amsmath}
\usepackage{amsthm}
\usepackage{latexsym}
\usepackage{amssymb}
\usepackage{xspace}
\usepackage{amscd}
\usepackage{enumerate}
\usepackage{lscape}
\usepackage{color}
\usepackage[normalem]{ulem}
\usepackage{amsmath}
\usepackage{amsmath, pb-diagram}
\usepackage[all]{xy}
\usepackage{amssymb}
\usepackage{amscd}
\newtheorem{theorem}{Theorem}[section]
 \newtheorem{cor}[theorem]{Corollary}
 \newtheorem{lemma}[theorem]{Lemma}
 \newtheorem{prop}[theorem]{Proposition}
 \theoremstyle{definition}

 \numberwithin{equation}{section}

\DeclareMathOperator{\End}{End}

\DeclareMathOperator{\Ker}{Ker}
\DeclareMathOperator{\im}{Im}

\begin{document}

\author{ Abyzov Adel Nailevich, Tran Hoai Ngoc Nhan
\\
and Truong Cong Quynh}
\address{Department of Algebra and Mathematical Logic, Kazan (Volga Region) Federal University, 18 Kremlyovskaya str., Kazan, 420008 Russia}
\email{aabyzov@ksu.ru, Adel.Abyzov@ksu.ru}
\address{Department of Mathematics, Danang University, 459 Ton Duc Thang, Danang city, Vietnam}
\email{tcquynh@dce.udn.vn, tcquynh@live.com}
\address{Department of Algebra and Mathematical Logic, Kazan (Volga Region) Federal University, 18 Kremlyovskaya str., Kazan, 420008 Russia}
\email{tranhoaingocnhan@gmail.com}
\title[Modules close to SSP- and SIP-modules]{Modules close to SSP- and SIP-modules}

\begin{abstract} In this paper, we investigate some properties  of
SIP, SSP and  CS-Rickart modules.  We give equivalent conditions for SIP and SSP modules; establish connections between the class of semisimple artinian rings and the class of SIP rings. It shows that $R$ is a semisimple artinian ring if and only if $R_R$ is SIP and every right $R$-module has a SIP-cover.  We also prove that $R$ is a semiregular ring and $J(R) = Z(R_R)$ if only if every finitely generated projective module is a SIP-CS module  which is also a $C2$ module.
\end{abstract}

\maketitle

\bigskip

\section{Introduction and notation}

Throughout this paper $R$ denotes an associative ring with identity, and modules will be unitary right $R$-modules. The Jacobson radical   ideal in $R$ is denoted by $J(R)$.   The notations $N \leq M$, $N \leq_e M$, $N \unlhd M$, or $N \subset_d M$ mean that $N$ is a submodule, an essential submodule, a fully invariant submodule, and  a direct summand of $M$,  respectively.   We refer to \cite{AF}, \cite{DHSW}, \cite{MM}, and \cite{W} for all the
undefined notions in this paper.

Recall that a module $M$ is called a \textit{SIP module} (respectively, \textit{SSP
module}) if the intersection (or the sum) of any two direct summands of $M$ is also a direct summand of
$M$ (see [\ref{Ga}, \ref{HS},  \ref{Wi}]). It is known that every Rickart right $R$-module $M$ (i.e., every endomorphism of $M$ has the kernel a direct summand)  has the SIP (see [\ref{GL1}, Proposition 2.16]) and every d-Rickart right $R$-module $M$ (i.e., every endomorphism of $M$ has the image  a direct summand) has the SSP  ([\ref{GL3}, Proposition 2.11]).

F. Karabacak and A. Tercan introduced the notion of  SIP-CS module in \cite{s2}. A module $M$ is called an \textit{SIP-CS module} if the intersection of any two direct summands of $M$
is essential in a direct summand of $M$. Note that SIP-CS modules are called SIP-extending modules
in \cite{s2}.  It is known that every CS-Rickart module has the SIP-CS (see [\ref{AB}, Proposition 1.(4)]).

In  this  paper, we provide some characterizations of SIP, SSP, SIP-CS and CS-Rickart modules.

\section{SIP modules and SSP modules}

Let $f: A\to B$ be a homomorphism. We denote by $\langle f \rangle$ the submodule of $A\oplus B$ as follows: $$\langle f \rangle=\{a+f(a)\ | \ a\in A\}.$$
 The following result is  obvious  and we can omit its proof.
\begin{lemma}Let $M=X\oplus Y$ and $f: A\to Y$, a homomorphism with $A\leq X$. Then
\begin{enumerate}
\item $A\oplus Y=\langle f \rangle\oplus Y$.
\item $\Ker(f)=X\cap \langle f \rangle$.
\end{enumerate}
\end{lemma}

We next study some properties of SIP and SSP modules via homomorphisms:

\begin{prop}\label{nssp} The following conditions are equivalent for a module $M$:
\begin{enumerate}
\item $M$ is SSP.
\item For any split monomorphism $f: A\to M$ with $A$ a direct summand of $M$, $A+\im(f)$ is a direct summand of $M$.
\item For any split epimorphism $f: M\to M/A$ with $A$ a direct summand of $M$, $A+\Ker(f)$ is a direct summand of $M$.
\end{enumerate}
\end{prop}
\begin{proof} $(1)\Rightarrow (2), (3)$ are obvious.

$(2)\Rightarrow (1)$. Assume that $M=A_1\oplus A_2$ and $f:A_1\to A_2$ an $R$-homomorphism. Call $T=\langle f \rangle$  a submodule of $M$ and hence  $M=T\oplus A_2$. We consider the homomorphism $\psi:A_1\to M$ given by $\psi(x)=x+f(x)$. It is easily  to see that $\psi$ is a split monomorphism. By (2), $A_1+\psi(A_1)=A_1+T$ is a direct summand of $M$. Furthermore, $A_1+T=A_1\oplus \im(f)$, which implies $\im(f)$ is a direct summand of $A_2$.

$(3)\Rightarrow (1)$.  Suppose  that $M=A_1\oplus A_2$ and $f:A_1\to A_2$ an $R$-homomorphism. Let  $T=\langle f \rangle$  be a submodule of $M$. Then  $M=T\oplus A_2$. Call  the homomorphism $\psi:M\to M/T$ given by $\psi(a_1+a_2)=a_2+T$ for all $a_1\in A_1, a_2\in A_2$. Clearly,  $\psi$ is a split epimorphism and $\Ker(\psi)=A_1$. By (3), $A_1+T$ is a direct summand of $M$. On the other hand, $A_1+T=A_1\oplus \im(f)$, which implies $\im(f)$ is a direct summand of $A_2$.
\end{proof}

\begin{cor}The following conditions are equivalent for a module $M$:
\begin{enumerate}
\item $M$ is SSP.
\item For any two direct summands $A_1$ and $A_2$ with $A_1\simeq A_2$, then   $A_1+A_2$ is a direct summand of $M$.
\end{enumerate}
\end{cor}

Similarly with SIP, we also have some characterizations of SIP-modules:

\begin{prop}The following conditions are equivalent for a module $M$:
\begin{enumerate}
\item $M$ is SIP.
\item For any split monomorphism $f: A\to M$ with $A$ a direct summand of $M$, $A\cap f(A)$ is a direct summand of $M$.
\item For any split epimorphism $f: M\to M/A$ with $A$ a direct summand of $M$, $A\cap \Ker(f)$ is a direct summand of $M$.
\end{enumerate}
\end{prop}

\begin{cor}The following conditions are equivalent for a module $M$:
\begin{enumerate}
\item $M$ is SIP.
\item For any two direct summands $A_1$ and $A_2$ with $A_1\simeq A_2$, then   $A_1\cap A_2$ is a direct summand of $M$.
\end{enumerate}
\end{cor}

\begin{prop} Let $R$ be a ring,  $M$ an $R$-$R$-bimodule and $T = R\propto M$
the corresponding trivial extension. The following conditions are equivalent:
\begin{enumerate}
\item  $T$  has the SSP
\item \begin{enumerate}
        \item  R has the SSP.
        \item  For every regular $x$ of $R$ with $x=xyx$,  we have $xM(1-xy) = 0$.
        \end{enumerate}
\end{enumerate}
\end{prop}
\begin{proof}
$(1)\Rightarrow (2).$ By \cite[Proposition 4.5]{G89}.

$(2)\Rightarrow (1).$ Assume that  $x=xyx$. For any $m\in M$, call   $z=xm(1-xy)$. It follows that  $z=(xy)z(1-xy)$. Note that $xy$ is idempotent of $R$. By \cite[Proposition 4.5]{G89}, $z=(xy)z(1-xy)=0.$
\end{proof}
Let $R$ be a ring and $\Omega$,  a class of right $R$-modules which is closed under
isomorphisms and summands. According to  Enochs in \cite{Enochs}, we study  the
notion of $\Omega$-envelope and the notion of $\Omega$-cover:

An $R$-homomorphism $g : M \to  E$ is called an {\it $\Omega$-envelope of a right $R$-module $M$}; if $E \in \Omega$ such that any diagram:
$$\begin{diagram}
\node{M}\arrow{s,t}{g'}\arrow{e,t}{g}\node{E}\arrow{sw,t,..}{h}\\
\node{E'}
\end{diagram}
$$
with $E'\in \Omega$, can be completed, and the diagram:
$$\begin{diagram}
\node{M}\arrow{s,t}{g}\arrow{e,t}{g}\node{E}\arrow{sw,t,..}{h}\\
\node{E}
\end{diagram}
$$
can be completed only by an automorphism $h$.

An $R$-homomorphism $g : E \to  M$ is called {\it an $\Omega$-cover of a right R-module $M$}; if $E \in \Omega$ such that any diagram:
$$\begin{diagram}
\node{E}\arrow{e,t}{g}\node{M}\\
\node[2]{E'}\arrow{nw,t,..}{h}\arrow{n,t}{g'}
\end{diagram}
$$
with $E'\in \Omega$, can be completed, and the diagram:
$$\begin{diagram}
\node{E}\arrow{e,t}{g}\node{M}\\
\node[2]{E}\arrow{nw,t,..}{h}\arrow{n,t}{g}
\end{diagram}
$$
can be completed only by an automorphism $h$.
\vskip 0.2cm

A right $R$-module $M$ is called a {\it $C3$-module} if
whenever $A$ and $B$ are direct summands  of $M$ with  $A\cap B=0$, then $A\oplus B$ is a direct summand of $M$. Dually, $M$ is called a {\it $D3$-module} if  whenever $M_{1}$ and $M_{2}$ are
direct summands of $M$ and \ $M=M_{1}+M_{2},$ then $M_{1}\cap M_{2}$ is a
direct summand of $M$.

\begin{prop}\label{sed3} The following conditions are equivalent for a ring $R$:

\begin{enumerate}
\item $R$ is a semisimple artinian ring.
\item Every right $R$-module has a D3-cover.
\item Every $2$-generated right $R$-module has a D3-cover.
\item Every right $R$-module has a D3-envelope.
\item Every $2$-generated right $R$-module has a D3-envelope.
\end{enumerate}
\end{prop}
\begin{proof}
$(1)\Rightarrow (2)\Rightarrow (3)$. Clear.

$(3)\Rightarrow (1)$. Let $S$ be  a simple  right  $R$-module. Call $\varphi: R_R\to S$ an epimorphism. By (3), $M=R_R \oplus S$   has a D3-cover, say $\alpha : C \rightarrow M$ where $C$ is a D3-module.   Let $ \iota_1: S\rightarrow M$  and $\iota_2: R_R\rightarrow M$ be the inclusion maps for all $i=1,2 $. Note that   $S$ and $R_R$ are   D3-modules, and  there are homomorphisms $\beta_1: S\rightarrow C, \beta_2: R_R\rightarrow C$  such that $\alpha\beta_i=\iota_i$.  Clearly, $id_M= \iota_1\oplus \iota_2=\alpha(\beta_1\oplus \beta_2)$. This shows that $M$  is isomorphic to a direct summand of $C$, which implies that  $M$  is a D3-module. We deduce that $\Ker(\varphi)$ is a direct summand of $R_R$ by \cite[Proposition 4]{YAI}. It follows that    $S$ is a projective module. Thus $R$ is semisimple.

$(1)\Rightarrow (4)\Rightarrow (5)$. Clear.

$(5)\Rightarrow (1)$. Let $S$ be a simple  right  $R$-module. Call $\varphi: R_R\to S$ an epimorphism. By (5), $M=R_R \oplus S$   has a D3-envelope,  named $\iota: M \rightarrow    E$ where   $E$ is a D3-module.   Since  $S$ and $R$ are D3-modules, there exist $f_1: E\rightarrow S,\ f_2: E\rightarrow R$ such that $f_i\iota=\pi_i$, where  $\pi_1: M\rightarrow S$ and $\pi_2: M\rightarrow R$ are   the projections. There exists $\phi: E\rightarrow M$ such that $\pi_i\phi=f_i$ for all $i=1,2$. It follows that $\phi\iota=id_M$, and hence  $\iota$  is a split monomorphism. Thus  $N \oplus E(N)$ is isomorphic to a direct summand of $E$. This gives that    $S \oplus R$ is also a D3-module.  We deduce that    $\Ker(\varphi)$ is a direct summand of $R_R$. So    $S$ is a projective module. Thus $R$ is semisimple.
\end{proof}

\begin{cor}
The following conditions are equivalent for a ring $R$:
\begin{enumerate}
\item $R$ is a semisimple artinian ring.
\item $R_R$ is SIP and every right $R$-module has a SIP-cover.
\item $R_R$ is SIP and every $2$-generated right $R$-module has a SIP-cover.
\item $R_R$ is SIP and every right $R$-module has a SIP-envelope.
\item $R_R$ is SIP and every $2$-generated right $R$-module has a SIP-envelope.
\end{enumerate}
\end{cor}
A ring $R$ is called a right {\it V-ring} if every simple right $R$-module is injective.

\begin{prop}\label{pro:2.8} The following conditions are equivalent for a ring $R$:
\begin{enumerate}
\item $R$ is a right  V-ring.
\item Every finitely cogenerated  right $R$-module has a C3-envelope.
\item Every finitely cogenerated  right $R$-module has a C3-cover.
\end{enumerate}
\end{prop}
\begin{proof} $(1)\Rightarrow (2), (3)$ are obvious.

$(2)\Rightarrow (1)$  Let $N$ be an arbitrary simple module. Assume that   $\iota: M= N \oplus E(N)\rightarrow    E$ is the C3-envelope, where   $E$ is a C3-module.   Since  $N$ and $E(N)$ are C3-modules, there exist $f_1: E\rightarrow N,\ f_2: E\rightarrow E(N)$ such that $f_i\iota=\pi_i$, where  $\pi_1: M\rightarrow N_i$ and $\pi_2: M\rightarrow E(N)$ are   the projections. There exists $\phi: E\rightarrow M$ such that $\pi_i\phi=f_i$ for all $i=1,2$. It follows that $\phi\iota=id_M$, and so  the monomorphism $\iota$  splits. Thus  $N \oplus E(N)$ is isomorphic to a direct summand of $E$. It follows that   $N \oplus E(N)$ is also a C3-module.  Therefore  $N$ is a direct summand of $E(N)$. This gives  $N$ is injective. Thus $R$ is a right V-ring.

$(3)\Rightarrow (1)$ The proof is similar to the proof $(3)\Rightarrow (1)$ of Proposition \ref{sed3}.
\end{proof}
Similarly, we also get the following result for injectivity of semisimple modules:
\begin{prop}The following conditions are equivalent for a ring $R$:
\begin{enumerate}
\item $R$ is a right Noetherian right V-ring .
\item Every  right $R$-module with essential socle  has a C3-envelope.
\item Every  right $R$-module with essential socle  has a C3-cover.
\end{enumerate}
\end{prop}

\section{SIP-CS modules}

A module $M$ is called \textit{relatively CS-Rickart to $N$} (or \emph{$N$-CS-Rickart}) if for every $\varphi\in \mathrm{End}_R(M, N)$, $\mathrm{Ker}\varphi$ is an essential submodule of a direct summand of $M$. A module $M$ is called \textit{relatively d-CS-Rickart to $N$} (or \emph{$N$-d-CS-Rickart}) if for every $\varphi\in \mathrm{End}_R(N, M)$, $\mathrm{Im}\varphi$ lies above a direct summand of $M$.  A module $M$ is called {\it CS-Rickart (d-CS-Rickart)} if $M$ is $M$-CS-Rickart (resp., $M$-d-CS-Rickart). $M$ is called a \textit{SIP-CS module} if
$A_i$ is essential in a direct summand of $M$ for all $i\in \mathcal{I}$, $\mathcal{I}$ is a finite index set, then $\bigcap_{i\in\mathcal{I}}A_i$ is essential in a direct summand of $M$ (see \cite{s2}). $M$ is called a  \textit{SSP-lifting  module} if
$A_i$ lies above a direct summand of $M$ for all $i\in \mathcal{I}$, $\mathcal{I}$ is a finite index set, then $\sum_{i\in\mathcal{I}}A_i$ lies above a direct summand of $M$ (see \cite{s3}). The class of CS-Rickart (d-CS-Rickart, SIP-CS, SSP-lifting) modules is studied by the authors in \cite{AT14, AB}.

\begin{lemma}\label{bd3:01} The following implications hold for a module $M=M_1\oplus \ldots \oplus M_n$:
 \begin{enumerate}
\item if $M$ is relatively CS-Rickart to $N$ then $M_i$ relatively CS-Rickart to $N$.
\item if $M$ is relatively d-CS-Rickart to $N$ then $M_i$ relatively d-CS-Rickart to $N$.
\end{enumerate}
\end{lemma}

\begin{proof}
We only need to prove for the case $n=2$, $i=1$.

(1) Assume that $M = M_1 \oplus M_2 $ is relatively CS-Rickart to $N$. There exists
$\varphi: M \to N$ such that $\varphi =\psi \oplus 0|_{M_2}$ for each $\psi: M_1 \to N$. By assumption, there exists  a direct summand $D$ of $M$ such that
$\mathrm{Ker}(\varphi)\leq_e D$. Since $\mathrm{Ker}(\varphi)=\mathrm{Ker}(\psi)\oplus M_2$ and $D = (D\cap M_1) \oplus
M_2$, it follows that $\mathrm{Ker} (\psi) \oplus M_2 \leq_e (D\cap M_1) \oplus M_2$.
Therefore $\mathrm{Ker}(\psi) \leq_e D \cap M_1 $. Since $D$ is a direct summand of $M$, $D \cap M_1$ is a direct summand of $M_1$. Hence $M_1$ is relatively CS-Rickart to $N$.

(2) Assume that $M = M_1\oplus M_2$  is relatively d-CS-Rickart to $N$. There exists
$\varphi: M \to N$ such that $\varphi = \psi \oplus 0 |_{M_2}$ for each $\psi: M_1 \to N$. By assumption, $\mathrm{Im}\varphi= \mathrm{Im}\psi$ lies above a direct summand of $N$. Thus, $M_1$ is relatively d-CS-Rickart to $N$.
\end{proof}

\begin{prop}\label{pro:reCIP} The following implications hold for a module $M$:
 \begin{enumerate}
\item if $M$ is a SIP-CS module with $C2$ condition and $M=M_1\oplus M_2$ then $M_1$  relatively CS-Rickart to $M_2$.
\item if $M$ is a SSP-lifting module with $D2$ condition and $M=M_1\oplus M_2$ then $M_1$  relatively d -CS-Rickart to $M_2$.
\end{enumerate}
\end{prop}

\begin{proof} Let  $f : M_1\to M_2$ be an $R$-homomorphism.  Then $M=\langle f\rangle\oplus M_2$.

$(1)$.  We have that $\Ker(f)=\langle f\rangle\cap M_1\leq_e eM$ for some $e^2=e\in S$, by $M$ is SIP-CS.
Let $\pi_1 : M_1\oplus M_2\to M_1$ be the canonical projection and hence  $eM\cap M_2=0$, implies that $\pi_1(eM)\cong eM$. Since $M$ is a $C2$ module, $\pi_1(eM)$ is a direct summand of $M$. Then, since $\Ker(f)\leq_e eM$, $\Ker(f)=\pi_1(\Ker(f))\leq_e \pi_1(eM)$. Hence,  $M_1$  is relatively CS-Rickart to $M_2$.

$(2)$.  We have that $\im(f)\oplus M_1=\langle f\rangle +M_1$ lies above $eM$ for some $e^2=e\in S$, by $M$ is SSP-lifting. Since $\langle f\rangle+M_1+M_2=M$, $eM+M_2=M$ by \cite[3.2.(1)]{JC}.  Let $\pi_1 : M_1\oplus M_2\to M_1$ be the canonical projection. Then ${\pi_1}_{\mid eM}  $   splits by $D2$ condition. It follows that $eM=(eM\cap M_2)\oplus N$ by $\Ker({\pi_1}_{\mid eM})=eM\cap M_2$. We have  $N\oplus M_2= N+(eM\cap M_2+M_2)=eM+M_2=M$ and obtain that   $N\oplus \im(f)=M_1\oplus\im(f)\supset eM$.

By modular law, $eM=N\oplus eM\cap\im(f)$. As $\dfrac{\im(f)\oplus M_1}{eM}\ll \dfrac{M}{eM}$, we have that $\dfrac{N\oplus \im(f)}{N\oplus eM\cap\im(f)}\ll \dfrac{N\oplus M_2}{N\oplus eM\cap\im(f)}$. This is equivalent to $\dfrac{\im(f)}{eM\cap\im(f)}\ll \dfrac{M_2}{eM\cap\im(f)}$, which implies that $\im(f)$ lies above the direct summand $eM\cap\im(f)$ of $M$.
\end{proof}

\begin{cor}\label{hq3:3}The following implications hold for a module $M=M_1\oplus \ldots \oplus M_n$:
 \begin{enumerate}
\item if $M$ is a SIP-CS module with $C2$ condition then $M_i$ is   relatively CS-Rickart to $M_j$ for every $i\neq j.$
\item if $M$ is a SSP-lifting with $D2$ condition then $M_i$ is  relatively d-CS-Rickart to $M_j$ for every $i\neq j.$
\end{enumerate}
\end{cor}
\begin{proof}
If $M$ is a SIP-CS module with $C2$ condition (respectively, SSP-lifting with $D2$ condition), then by Proposition \ref{pro:reCIP},  $\bigoplus_{i\not=j} M_i$  is  relatively CS-Rickart to $M_j$ (respectively, relatively d-CS-Rickart to $M_j$). By Lemma \ref{bd3:01},  $M_i$  is  relatively CS-Rickart to $M_j$ (respectively, relatively d-CS-Rickart to $M_j$) for every $i\neq j$.
\end{proof}

\begin{cor}\label{hq3:4}The following implications hold for a module $M$:
 \begin{enumerate}
\item if $M\oplus M$ is a SIP-CS module with $C2$ condition then $M$ is a CS-Rickart module.
\item if $M\oplus M$  is a SSP-lifting with $D2$ condition then $M$ is a d-CS-Rickart module.
\end{enumerate}
\end{cor}
\begin{proof}
Follow from Corollary \ref{hq3:3}.
\end{proof}

The singular submodule $Z(M)$ of a right $R$-module $M$ is defined as $Z(M)=\{m\in M: ann^{r}_R(m)$ is an essential right ideal of $R\}$ where $ann^r_R(m)$ denotes the right annihilator of $m$ in $R$.
The singular submodule of $R_R$ is called the (right) singular
ideal of the ring $R$ and is denoted by $Z(R_R)$. It is well known that $Z(R_R)$ is
indeed an ideal of $R$.

Next we give a necessary and sufficient condition for a ring over which every finitely generated projective module to be a SIP-CS - module which is also a $C2$ module.
\begin{theorem}\label{3.2} The following conditions are equivalent for a ring $R$:
\begin{itemize}
 \item[(1)] $R$ is a semiregular ring and $J(R) = Z(R_R)$;
 \item[(2)] Every finitely generated projective module is a CS-Rickart module which is also a $C2$ module.
 \item[(3)] Every finitely generated projective module is a SIP-CS module which is also a $C2$ module.
 \item[(4)] Every finitely generated projective module is a SIP-CS module which is also a $C3$ module.

\end{itemize}
\end{theorem}

\begin{proof}

$(1)\Rightarrow (2)$. Follows from  \cite[Theorem 2]{AB}.

$(2)\Rightarrow (3)$. Follows from  \cite[Proposition 1]{AB}.

$(3)\Rightarrow (2)$. Let $P$ be a finitely generated projective module. By the hypothesis, $P$ is a SIP-CS module which is also a $C2$ module. Then $P\oplus P$ is a SIP-CS module which is also a $C2$ module. Since Proposition \ref{pro:reCIP}, $P$ is  relatively CS-Rickart to $P$,  it means that $P$ is a CS-Rickart module.

$(3)\Leftrightarrow (2)$. Follows from  \cite[Corollary 3.5]{AT14}.
\end{proof}

\begin{lemma}\label{lem:01} The following conditions are equivalent for a module $M$.
 \begin{enumerate}
\item $M$ is a SIP-CS module.
\item Intersection of every pair of direct summands of $M$ is essential in a direct summand
of M.
\end{enumerate}
\end{lemma}
\begin{proof}It is obvious.
\end{proof}

\begin{prop}\label{pro:ker} Assume that  $M$ is a SIP-CS module. Then for any decomposition   $M=M_1\oplus M_2$  and   $f: M_1\to M_2$ is a homomorphism, then $\Ker(f)$ is essential  in  a direct summand of $M$.
\end{prop}
\begin{proof} Assume that $M = M_1 \oplus  M_2$ and  $f : M_1\to M_2$ an $R$-homomorphism.
Call $T= \langle f\rangle$ a submodule of $M$. So $M=T\oplus M_2$ and $\Ker(f)=T\cap M_1$. On the other hand,  by the hypothesis, $M$ is a SIP-CS and hence   $\Ker(f)$ is essential in  a direct summand of $M$ by Lemma \ref{lem:01}.
\end{proof}
\begin{cor}\label{cor:kerd}Let $M$ be a  module and  $N$,  a nonsingular module. If $M\oplus N$ is a SIP-CS module, then every  homomorphism from $M$ to $N$  has the kernel a direct summand of $M$.
\end{cor}
\begin{proof} Let $f: M\to N$ be a non-zero homomorphism. By  Proposition \ref{pro:ker}, $\Ker(f)$ is essential in a direct summand of  $M\oplus N$. Assume that   $A$ is a direct summand of  $M\oplus N$ such that  $\Ker(f)\leq_e A$. Call  $\pi_M: M\oplus N\to M$  the canonical projection and $h=(f\circ \pi_M)|_A: A\to N$. Therefore   $\Ker(h)=\Ker(f)\oplus (N\cap A)$. We have that  $\Ker(f)\leq_e A$ and obtain that $\Ker(f)\oplus (N\cap A)\leq_e A$. It follows that $\Ker(f)\oplus (N\cap A)= A$. Thus  $\Ker(f)$ is a direct summand of $M$.
\end{proof}
\begin{cor}Let $M$ be an indecomposable module and  $N$ be a nonsingular module. If $M\oplus N$ is a SIP-CS module, then every nonzero homomorphism from $M$ to $N$ is a monomorphism.
\end{cor}
\begin{prop} Let $M$ be a nonsingular right $R$-module. If $(R\oplus M)_R$ is a  SIP-CS module, then every cyclic submodule of $M$ is projective.
\end{prop}
\begin{proof} Let $m$ be a non-zero arbitrary element of $M$. Call the homomorphism $\varphi: R_R\to M$ given by $\varphi(x)=mx$. As $(R\oplus M)_R$ is a  SIP-CS module,  $\Ker(\varphi)$  is a direct summand of $R_R$ by Corollary \ref{cor:kerd}.  It follows that $\im(\varphi)$ is isomorphic to a direct summand of $R_R$. Thus $mR$ is a projective module.
\end{proof}
A ring $R$ is called  right {\it (semi)hereditary} if every (finitely generated) right ideal of $R$   is projective.

\begin{theorem}\label{thm:heredi}  The following statements are equivalent for a right nonsingular  ring $R$.
\begin{enumerate}
\item  $R$   is  right hereditary.
\item  Every  projective  right $R$-module is  a SIP-CS module.
\end{enumerate}
\end{theorem}
\begin{proof} $(1)\Rightarrow (2)$ is obvious.

$(2)\Rightarrow (1)$ Let $I$ be a  right ideal  of $R$. We will show that $I$ is a projective module. Call an epimorphism $\varphi: F\to  N$ for some free right $R$-module $F$.    Let  $\iota$  be the inclusion map from $I$ to $R_R$. Consider the homomorphism  $\iota\circ \varphi: F\to R_R$. By
(2),  $F\oplus  R_R$  is a  SIP-CS module. We have from  Corollary \ref{cor:kerd}, $\Ker(\varphi)=\Ker(\iota\circ \varphi)$  is  a direct summand $F$. This gives that  $F=\Ker(\varphi)\oplus B$ for some submodule $B$ of $F$.   Thus, $I$ is projective.
\end{proof}

The author Warfied proved that if $R$ is right  serial, then $R$ is right  nonsingular if and only if $R$ is right  semihereditary.

The same argument of the proof of Theorem \ref{thm:heredi}, we also have the following result of semihereditary rings:

\begin{theorem}\label{thm:semiheredi}  The following statements are equivalent for a right nonsingular  ring $R$.
\begin{enumerate}
\item  $R$   is  right semihereditary.
\item  Every  finitely generated projective right  $R$-module is  a SIP-CS module.
\item  Every  finitely generated free  right $R$-module is  a SIP-CS module.
\end{enumerate}
\end{theorem}

Let $M$ be a right $R$-module and  $S=\End(M)$. We denote $$\Delta(S)=\{f\in S\ |\ \Ker(f)\leq_e M\}.$$

An $R$-module is called a {\it self-generator}  if it generates all its submodules.

\begin{theorem}\label{csn} The following conditions are equivalent for a self-generator module  $M$ with $S=\End(M)$:
\begin{enumerate}
\item $S$ is a semiregular ring with $J(S)=\Delta(S)$.
\item $M$ is a CS-Rickart and C2 module.
\end{enumerate}
\end{theorem}
\begin{proof} $(1)\Rightarrow (2)$ Assume that  $S$ is a semiregular ring with $J(S)=\Delta(S)$. As $M$ is a self-generator, $J(S)=\Delta(S)\leq Z(S_S)$. We deduce that $S$ is  right C2.  This gives  that $M$ is a C2-module by \cite[Theorem 7.14(1)]{NY1}. Let $\alpha: M\to M$ be an endomorphism of $M$. As $S$ is a semiregular ring, there exists $\beta\in S$ such that $\beta=\beta\alpha\beta$ and $\alpha-\alpha\beta\alpha\in J(S)$ by \cite[Theorem 2.9]{Ni}.  Call $e=1-\beta\alpha$. Then $e^2=e\in S$. As $\alpha-\alpha\beta\alpha\in \Delta(S)$, $\Ker(\alpha-\alpha\beta\alpha)\leq_eM$ and hence $\Ker(\alpha-\alpha\beta\alpha)\cap e(P)\leq_ee(M)$. It is easily to check that $\Ker(\alpha-\alpha\beta\alpha)\cap e(M)=\Ker(\alpha)$. We deduce that $\Ker(\alpha)\leq_e e(M)$.

$(2)\Rightarrow (1)$ By \cite[Theorem 3.2]{QKT}.
\end{proof}

\begin{cor}The following conditions are equivalent for a self-generator module $M$ with $S=\End(M^{(\mathbb{N})})$:
\begin{enumerate}
\item $S$ is a semiregular ring with $J(S)=\Delta(S)$.
\item $M^{(\mathbb{N})}$ is a CS-Rickart and C2 module.
\item $M^{(\mathbb{N})}$ is a SIP-CS and C2 module.
\end{enumerate}
\end{cor}
\begin{proof}By Proposition \ref{pro:reCIP} and  Theorem \ref{csn}.
\end{proof}
\section*{Acknowledgements}The third author was supported by the Russian Government Program of Competitive Growth of Kazan Federal University. The author would like to thank the members of Department of Algebra and Mathematical Logic, Kazan (Volga Region) Federal University for their hospitality.
\bibliographystyle{amsplain}

\end{document}